\documentclass[12pt]{amsart}
\usepackage{fullpage}
\usepackage{times}
\usepackage{latexsym}
\usepackage{amssymb}
\usepackage[all]{xy}

\newtheorem{thm}{Theorem}[section]
\newtheorem{prop}[thm]{Proposition}
\newtheorem{cor}[thm]{Corollary}

\newtheorem{lem}[thm]{Lemma}

\theoremstyle{remark}
\newtheorem{rem}[thm]{Remark}
\newtheorem{defn}{Definition}[section]

\usepackage{hyperref}

\newcommand{\CP}{\mathbb{CP}}

\newcommand{\Q}{\mathbb{Q}}
\newcommand{\R}{\mathbb{R}}
\newcommand{\Z}{\mathbb{Z}}

\newcommand{\SO}{\mathrm{\SO}}

\title[Ordering Thurston's geometries by maps of non-zero degree]{Ordering Thurston's geometries by maps of non-zero degree}

\author{Christoforos Neofytidis}
\address{Department of Mathematical Sciences, {\smaller SUNY} Binghamton, Binghamton, NY 13902-6000, USA}
\email{chrisneo@math.binghamton.edu}
\date{\today}

\begin{document}

\maketitle

\begin{abstract}
We obtain an ordering of closed aspherical $4$-manifolds that carry a non-hyperbolic Thurston geometry. As application, we derive that the Kodaira dimension of geometric $4$-manifolds is monotone with respect to the existence of maps of non-zero degree.
\end{abstract}

\section{Introduction}

The existence of a map of non-zero degree defines a transitive relation, called domination relation, on the homotopy types of closed oriented manifolds of the same dimension. Whenever there is a map of non-zero degree $M\longrightarrow N$ we say that $M$ {\em dominates} $N$ and write $M\geq N$. In general, the domain of a map of non-zero degree is a more complicated manifold than the target.

Gromov suggested studying the domination relation as defining an ordering of compact oriented manifolds of a given dimension; see~\cite[pg. 1]{CarlsonToledo}. In dimension two, this relation is a total order given by the genus. Namely, a surface of genus $g$ dominates another surface of genus $h$ if and only if $g\geq h$. However, the domination relation is not generally an order in higher dimensions, e.g. $S^3$ and $\mathbb{RP}^3$ dominate each other but are not homotopy equivalent. Nevertheless, it can be shown that the domination relation is a partial order in certain cases. For instance, $1$-domination defines a partial order on the set of closed Hopfian aspherical manifolds of a given dimension (see~\cite{Rong} for $3$-manifolds). Other special cases have been studied by several authors; see for example~\cite{CarlsonToledo,Bel,BBM,Zhang}.

Wang~\cite{Wang:3-mfdsasp} obtained an ordering of all closed aspherical $3$-manifolds in a reasonable sense, according to Thurston's geometrization picture. We will discuss Wang's result in Section \ref{s:Wang}, together with an extension of that result to non-aspherical $3$-manifolds obtained in~\cite{KotschickNeofytidis}; cf. Theorem \ref{t:order3-mfds}. In this paper, our goal is to order in the sense of Wang all non-hyperbolic closed $4$-manifolds that carry a Thurston aspherical geometry:

\begin{thm}\label{t:order4}
 Consider all closed oriented $4$-manifolds that possess a non-hyperbolic aspherical geometry. If there is an oriented path from a geometry $\mathbb{X}^4$
to another geometry $\mathbb{Y}^4$ in Figure \ref{d:nonhypmaps}, then any closed $\mathbb{Y}^4$-manifold is dominated by a closed $\mathbb{X}^4$-manifold.
If there is no oriented path from $\mathbb{X}^4$ to $\mathbb{Y}^4$, then no closed $\mathbb{X}^4$-manifold dominates a closed $\mathbb{Y}^4$-manifold.
\end{thm}

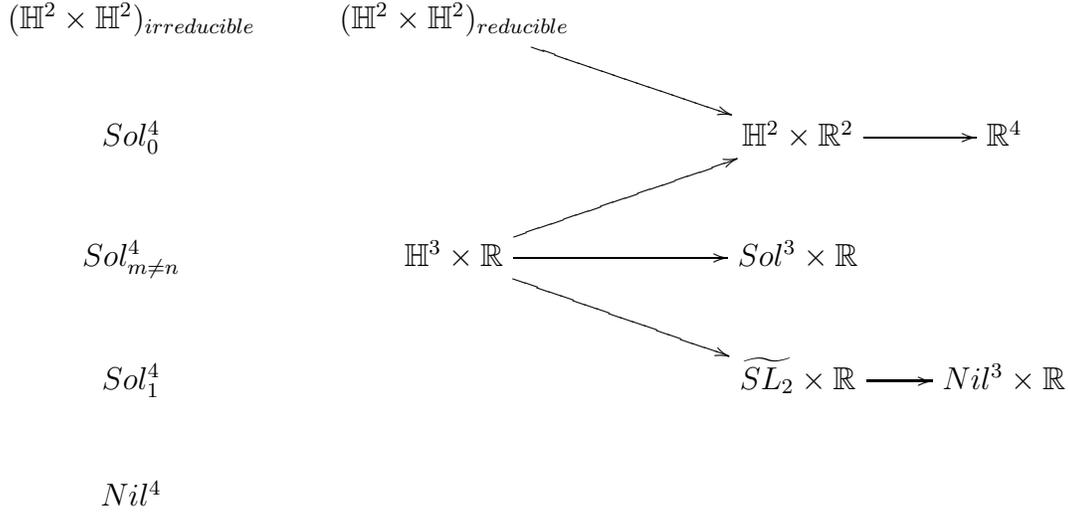
\begin{figure}[ht!]
    \[
\xymatrix{
 &(\mathbb{H}^2 \times \mathbb{H}^2)_{irreducible}  & (\mathbb{H}^2 \times \mathbb{H}^2)_{reducible} \ar[rrd]& & & \\
& Sol_0^4 &   &                                       &  \mathbb{H}^2 \times \R^2 \ar[r] & \R^4\\
&Sol_{m\neq n}^4 & 
 \mathbb{H}^3 \times \R \ar[rr] \ar[rru] \ar[rrd] & & Sol^3
\times \R & &\\
&Sol_1^4 & & & \widetilde{SL_2} \times \R \ar[r] & Nil^3 \times \R\\
 & Nil^4  &  & & &}
\] 
\caption{\small{Ordering the non-hyperbolic aspherical Thurston geometries in dimension four.}}
\label{d:nonhypmaps}
\end{figure} 

We will not be dealing with the two hyperbolic geometries (real and complex), partially because some of the results concerning those geometries are well-known and because the domination relation for those geometries has been studied by other authors; see Section \ref{ss:hyper} for a brief discussion. Similarly, the non-aspherical geometries are not included in the above theorem. Those geometries are either products or their representatives are simply connected; see~\cite{Neobranch, Neothesis} for a discussion.

An important question in topology (see for example~\cite{MT}) is whether a given numerical invariant $\iota$ is monotone with respect to the domination relation, that is, whether $M\geq N$ implies $\iota(M)\geq\iota(N)$. The Kodaira dimension is a significant invariant in the classification scheme of manifolds; we refer to~\cite{Li} for a recent survey on the various notions of Kodaira dimension of low-dimensional manifolds.
Zhang~\cite{Zhang} defined the Kodaira dimension $\kappa^t$ for $3$-manifolds and showed (\cite[Theorem 1.1]{Zhang}) that it is monotone with respect to the domination relation; 
see Theorem \ref{t:Zhang}. Moreover, Zhang~\cite{Zhang} defined the Kodaira dimension $\kappa^g$ for geometric $4$-manifolds and suggested that it should also be monotone with respect to the domination relation (we will give both definitions of $\kappa^t$ and $\kappa^g$ in Section \ref{s:Kodaira}). In this paper we confirm Zhang's suggestion:

\begin{thm}\label{t:Kodaira4}
Let $M, N$ be closed oriented geometric $4$-manifolds. If $M\geq N$, then $\kappa^g(M)\geq\kappa^g(N)$.
\end{thm}

\subsection*{Outline} 
In Section \ref{s:Wang} we discuss briefly Wang's ordering of $3$-manifolds. In Section \ref{s:4-geom} we list Thurston's aspherical geometries in dimension four together with some properties of manifolds modeled on those geometries. In Section \ref{s:products} we discuss maps between $4$-manifolds that are finitely covered by direct products, extending in particular Wang's ordering to $4$-manifolds that are virtual products of type $N\times S^1$, and in Section \ref{s:proofend} we complete the proof of Theorem \ref{t:order4}. In Section \ref{s:Kodaira} we discuss the relationship between the domination relation and Kodaira dimensions.

\subsection*{Acknowledgments} 
I am grateful to D. Kotschick under whose wise guidance and continuous support this project was carried out. Also, I would like to thank C. L\"oh and S. Wang for useful comments and discussions and W. Zhang for bringing to my attention the relationship of this work to Kodaira dimensions. 
The financial support of the {\em Deutscher Akademischer Austausch Dienst} (DAAD) is also gratefully acknowledged.

\section{Wang's ordering in dimension three}\label{s:Wang}

\subsection{Ordering 3-manifolds}

Let $M$ be a closed aspherical $3$-manifold that does not possess one of the six aspherical Thurston geometries. Then there is a finite family of splitting tori so that $M$ can be cut into pieces, called JSJ pieces (after Jaco-Shalen and Johannson). We call $M$ {\em non-trivial graph} manifold if all the JSJ pieces are Seifert. If there is a non-Seifert JSJ piece, then this piece must be hyperbolic by Perelman's proof of Thurston's geometrization conjecture. In that case, we call $M$ {\em non-graph} manifold. 

 In~\cite{Wang:3-mfdsasp} Wang suggested an ordering of all closed $3$-manifolds.  According to Wang's work and to the results of~\cite{KotschickNeofytidis} 
 we have the following:

\begin{thm}[Wang's ordering \cite{Wang:3-mfdsasp,KotschickNeofytidis}]\label{t:order3-mfds}
 Let the following classes of closed oriented $3$-manifolds: 
\begin{itemize}
 \item[\normalfont{(i)}] aspherical geometric: modeled on one of the geometries $\mathbb{H}^3$, $Sol^3$, $\widetilde{SL_2}$, $\mathbb{H}^2 \times \R$,
$Nil^3$ or $\R^3$;
 \item[\normalfont{(ii)}] aspherical non-geometric: $\mathrm{(GRAPH)}$ non-trivial graph or $\mathrm{(NGRAPH)}$ non-geometric irreducible non-graph;
 \item[\normalfont{(iii)}] finitely covered by $\#_p(S^2 \times S^1)$, for some $p\geq 0$\footnote{As explained in~\cite{KotschickNeofytidis}, these manifolds constitute the class of {\em rationally inessential} $3$-manifolds, namely closed oriented $3$-manifolds whose classifying map of the universal covering is trivial in rational homology of degree $3$.}.
\end{itemize}
If there exists an oriented path from a class $X$ to another class $Y$ in Figure \ref{f:order3-mfds}, then any representative in the class $Y$ is dominated
by some representative of the class $X$. If there is no oriented path from the class $X$ to the class $Y$, then no manifold in the class $Y$ can be
dominated by a manifold of the class $X$. 
\end{thm}

\begin{figure}[]
     \[
\xymatrix{
& & & & \mathbb{H}^2 \times \R \ar[rrr] \ar[rrrd] & & & \R^3 \ar[d]^{(p\leq 1)}\\
\mathbb{H}^3 \ar[r] & \mathrm{(NGRAPH)} \ar[l] \ar[r] &
\mathrm{(GRAPH)} 
\ar[rru] \ar[rrd] \ar[rr] & &
Sol^3 \ar[rrr]^{(p\leq 1)} & & & \#_p(S^2 \times S^1)\\
& & & &\widetilde{SL_2} \ar[rrr] \ar[rrru] & & & Nil^3 \ar[u]_{(p\leq 1)}}
\]
\caption{\small{Ordering $3$-manifolds by maps of non-zero degree~\cite{Wang:3-mfdsasp,KotschickNeofytidis}.}}
\label{f:order3-mfds}
\end{figure}
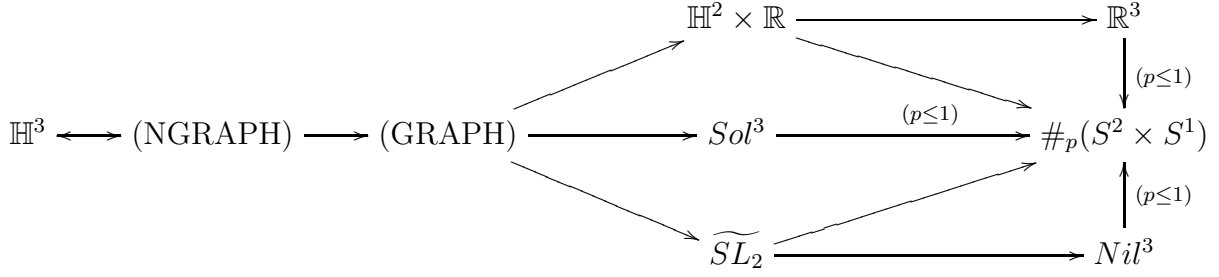

The proof of Theorem \ref{t:order3-mfds} concerning maps between aspherical $3$-manifolds is given in Wang's paper~\cite{Wang:3-mfdsasp}. Concerning maps from $\mathbb{H}^2 \times \R$-manifolds to $\widetilde{SL_2}$- or $Nil^3$-manifolds, or when the target manifold is finitely covered by $\#_p(S^2\times S^1)$ the proof is given in~\cite{KotschickNeofytidis}. As pointed out in~\cite{Wang:3-mfdsasp}, some of the non-existence results of the above theorem can be deduced using well-known tools, such as Gromov's simplicial volume, Thurston's norm and the Seifert volume. However, the proofs in~\cite{Wang:3-mfdsasp} do not rely on this machinery and are given in a uniform elementary way, using only properties of the fundamental groups and incompressible surfaces of 3-manifolds. 

In fact, there is one case which is not explicitly treated neither in~\cite{Wang:3-mfdsasp} nor in~\cite{KotschickNeofytidis}, namely, that $S^2 \times S^1$ admits a non-zero degree map from a $Sol^3$-manifold. Recall that every closed $Sol^3$-manifold is virtually a mapping torus of $T^2$ with hyperbolic monodromy~\cite{Scott:3-mfds}. Therefore, a map for this remaining case can be obtained similarly to the construction given in~\cite[Prop. 6]{KotschickNeofytidis}, where it was shown that a $T^2$-bundle over $S^1$ with monodromy 
$\left( \begin{array}{cc}
  1 & 1 \\
  0 & 1 \\
\end{array} \right)$ dominates $S^2\times S^1$. We only need to replace the monodromy of the mapping torus of $T^2$ by a hyperbolic one, for example by
$\left( \begin{array}{cc}
  2 & 1 \\
  1 & 1 \\
\end{array} \right)$; see the proof of~\cite[Prop. 6]{KotschickNeofytidis} for details.

Furthermore, we note that the restriction $p \leq 1$ for the arrows from $Sol^3$, $Nil^3$ and $\R^3$ to $\#_p (S^2 \times S^1)$ in Figure
\ref{f:order3-mfds} is required for the following reasons:
\begin{itemize}
 \item[($Sol^3$)] The first Betti number of closed $Sol^3$-manifolds is one and so these manifolds cannot dominate $\#_p (S^2 \times S^1)$ when $p \geq 2$ (recall that a map of non-zero degree induces epimorphisms in rational homology).
 \item[($Nil^3$)] Let $M$ be a closed $Nil^3$-manifold. After passing to a finite covering, if necessary, we may assume that $M$ is a non-trivial circle bundle over $T^2$~\cite{Scott:3-mfds}. Suppose that there is a continuous map $f \colon M \longrightarrow \#_p (S^2 \times S^1)$, where $p \geq 2$, such that
$\pi_1(f)(\pi_1(M))$ is a finite index
subgroup of $\pi_1(\#_k(S^2\times S^1))=F_p$, i.e. $\pi_1(f)(\pi_1(M)) = F_l$ for some $l\geq p\geq 2$. The homomorphism $\pi_1(f) \colon \pi_1(M) \longrightarrow F_l$ must factor
through $\pi_1(T^2) = \Z^2$, because free groups on more than one generators do not have center, whereas $C(\pi_1(M))=\Z$. However, $\Z^2$ cannot surject
onto such a free group and so $f$ must be of zero degree.
\item[($\R^3$)] A flat $3$-manifold has fundamental group virtually $\Z^3$, which cannot surject onto a free group on more than one generators.
\end{itemize}

\begin{rem}
Note that if we restrict to the class of closed oriented aspherical $3$-manifolds and to degree one maps, then the domination
relation defines a partial order on those manifolds, because $3$-manifold groups are Hopfian 
(i.e. every surjective endomorphism is an isomorphism), by Perelman's proof of the geometrization conjecture. For further details, we refer to the works of
Wang~\cite{Wang:3-mfdsasp,Wang:inj} and Rong~\cite{Rong}. 
\end{rem}

\subsection{A remark about $\#(S^2\times S^1)$}
In~\cite{KotschickNeofytidis}, it was shown that a connected sum $\#_p(S^2\times S^1)$ is dominated by both a non-trivial circle bundle over a closed oriented surface $\Sigma_p$ (of genus $p$), and by the product $\Sigma_p\times S^1$. An interesting observation is that the genus $p$ is the smallest possible:

\begin{lem}\label{l:sharpgenus}
 Let $M$ be a circle bundle over a closed oriented surface $\Sigma_g$. If $M\geq\#_p (S^2 \times S^1)$, then $g \geq p$. 
\end{lem}
\begin{proof}
 The interesting cases occur when $p \geq 2$. Suppose that $f \colon M \longrightarrow \#_p (S^2 \times S^1)$ is a map of non-zero degree. Then 
the base surface $\Sigma_g$ of $M$ is aspherical and $\pi_1(f)(\pi_1(M))$ is a free group on $l \geq p$ generators. The infinite cyclic group generated by
the circle fiber of $M$ is central in $\pi_1(M)$, and therefore is mapped trivially in $F_l$, which means that $\pi_1(f)$ factors through
$\pi_1(\Sigma_g)$. Since the degree of $f$ is not zero, we obtain
an injective homomorphism 
\[
 H^1(f) \colon H^1(F_l) \longrightarrow H^1(\Sigma_g).
\]
(Note that both $H^1(F_l)$ and $H^1(\Sigma_g)$ are torsion-free.) The cup product of any two elements $\alpha_1,\alpha_2$ in $H^1(F_l)$ is trivial, because
$H^2(F_l) = 0$. By the naturality of the cup product, we have that $H^1(f)(\alpha_1) \cup H^1(f)(\alpha_2)$ vanishes as well.
This implies that $l \leq \frac{1}{2} \dim H^1(\Sigma_g) = g$, because otherwise the intersection form of $\Sigma_g$ would be degenerate.
\end{proof}

 \section{The $4$-dimensional aspherical geometries}\label{s:4-geom}
  
In this section we enumerate Thurston's aspherical geometries in dimension four and give some properties that we will need for our proofs.

The $4$-dimensional Thurston's geometries were classified by Filipkiewicz~\cite{Filipkiewicz}. We list the aspherical geometries that are realized by compact manifolds, following Wall's papers~\cite{Wall:geom4-mfds1} and~\cite{Wall:geom4-mfds2}. 
This list (Table \ref{table:4geom}) will be used as an organizing principle for the proof of Theorem \ref{t:order4}.

\begin{table}
\centering
{\small
\begin{tabular}{c|c}
Type & Geometry $\mathbb{X}^4$\\
\hline
& \\
          & $\mathbb{H}^3\times\mathbb{R}$, $Sol^3\times\R$, \\
Product   & $\widetilde{SL_2}\times\mathbb{R}$, $Nil^3\times\mathbb{R}$, \\
          & $\mathbb{H}^2\times\mathbb{R}^2$, $\R^4$, \\
          & $\mathbb{H}^2\times\mathbb{H}^2$ \\&\\
Solvable    & $Nil^4$, \\
non-product & $Sol^4_{m \neq n}$, $Sol^4_0$,\\
            & $Sol^4_1$\\&\\
Hyperbolic & $\mathbb{H}^4$, $\mathbb{H}^2(\mathbb{C})$\\&\\            
\end{tabular}}
\vspace{9pt}
\caption{{\small The $4$-dimensional aspherical Thurston geometries with compact representatives.}}\label{table:4geom}
\end{table}

\medskip

\noindent{{\bf Product geometries.}} Seven of the aspherical geometries are products of lower dimensional geometries: $\mathbb{H}^3 \times \R$,
$Sol^3 \times \R$,
$\widetilde{SL_2} \times \R$, $Nil^3 \times \R$, $\mathbb{H}^2 \times \R^2$, $\R^4$ and $\mathbb{H}^2 \times \mathbb{H}^2$. 

Closed manifolds possessing a geometry of type $\mathbb{X}^3 \times \R$ satisfy the following property:

\begin{thm}[\normalfont{\cite[Sections 8.5 and 9.2]{Hillman}}]\label{t:hillmanproducts}
 Let $\mathbb{X}^3$ be a $3$-dimensional aspherical geometry. A closed $4$-manifold possessing the geometry $\mathbb{X}^3 \times \R$ is finitely
covered by a product $N \times S^1$, where $N$ is a closed oriented $3$-manifold carrying the geometry $\mathbb{X}^3$.
\end{thm}

The geometry $\mathbb{H}^2 \times \mathbb{H}^2$ can be realized both by manifolds that are virtual products of two closed hyperbolic
surfaces and by manifolds that are not even (virtual) surface bundles. These two types are known as the {\em reducible} and the {\em irreducible}
$\mathbb{H}^2 \times \mathbb{H}^2$ geometry respectively; see~\cite[Section 9.5]{Hillman} for further details and characterizations of those geometries.

\medskip

\noindent{{\bf Solvable non-product geometries.}} There exist four aspherical non-product geometries of solvable type. Below, we describe their
model Lie groups together with some characterizations of manifolds modeled on each of those geometries.

The nilpotent Lie group $Nil^4$ is defined as the semi-direct product $\R^3 \rtimes \R$, where $\R$ acts on $\R^3$ by
\[
t \mapsto 
\left(\begin{array}{ccc}
   1 & e^t & 0 \\
   0 & 1 & e^t \\
   0 & 0 & 1   \\
\end{array} \right).
\]

A closed $Nil^4$-manifold is characterized by the following:

\begin{prop}{\normalfont (\cite[Prop. 6.10]{NeoIIPP}).}\label{p:nil4}
A closed $Nil^4$-manifold $M$ is a virtual circle bundle over a closed oriented $Nil^3$-manifold and the center of $\pi_1(M)$ remains infinite cyclic in finite covers.
\end{prop}

The model spaces for the three non-product solvable -- but not nilpotent -- geometries are defined as follows:

Let $m$ and $n$ be positive integers and $a > b > c$ reals such that $a+b+c=0$ and $e^a,e^b,e^c$ are
roots of the equation $P_{m,n}(\lambda)=\lambda^3-m\lambda^2+n\lambda-1=0$. If $m \neq n$, the Lie group $Sol_{m \neq n}^4$ is defined as $\R^3 \rtimes
\R$, where $\R$ acts on $\R^3$ by
\[
t \mapsto 
\left(\begin{array}{ccc}
   e^{at} & 0 & 0 \\
   0 & e^{bt} & 0 \\
   0 & 0 & e^{ct} \\
\end{array} \right).
\]
We remark that the case $m=n$ gives $b = 0$ and corresponds to the product geometry $Sol^3 \times \R$. 

If we require two equal roots of the polynomial $P_{m,n}$, then we obtain the model space of the $Sol_0^4$ geometry, again defined as $\R^3
\rtimes \R$, where now the action of $\R$ on $\R^3$ is given by
\[
t \mapsto 
\left(\begin{array}{ccc}
   e^{t} & 0 & 0 \\
   0 & e^{t} & 0 \\
   0 & 0 & e^{-2t} \\
\end{array} \right).
\]

It was shown in~\cite{KotschickLoeh1} that aspherical manifolds (more generally, rationally essential manifolds) are not dominated by direct products if their fundamental group is {\em not presentable by products}. A group $\Gamma$ is not presentable by products if for every homomorphism $\varphi\colon\Gamma_1\times\Gamma_2\longrightarrow\Gamma$ with $[\Gamma:\mathrm{im}(\varphi)]<\infty$, one of the images $\varphi(\Gamma_i)$ is finite. Manifolds modeled on one of the geometries $Sol_{m\neq n}^4$ or $Sol_0^4$ fulfill the latter property:

\begin{prop}{\normalfont (\cite[Prop. 6.13]{NeoIIPP}).}\label{p:sol&sol}
The fundamental group of a closed $4$-manifold possessing one of the geometries $Sol_{m\neq n}^4$ or $Sol_0^4$ is not presentable by products.
\end{prop}

The last solvable model space is an extension of $\R$ by the $3$-dimensional Heisenberg group
$
 Nil^3 = 
\Biggl\{ \left( \begin{array}{ccc}
  1 & x & z \\
  0 & 1 & y \\
  0 & 0 & 1 \\
\end{array} \right) \biggl\vert
\ x,y,z \in \R \Biggl\}.
$
Namely, the Lie group $Sol_1^4$ is defined as the semi-direct product $Nil^3 \rtimes \R$, where $\R$ acts on $Nil^3$ by
\[
t \mapsto 
\left(\begin{array}{ccc}
   1 & e^{-t}x & z \\
   0 & 1 & e^{t}y \\
   0 & 0 & 1 \\
\end{array} \right).
\]

Manifolds modeled on this geometry fulfill the following property:

\begin{prop}{\normalfont \cite[Prop. 6.15]{NeoIIPP}).}\label{p:sol1}
A closed $Sol_1^4$-manifold $M$ is a virtual circle bundle over a mapping torus of $T^2$ with hyperbolic monodromy (i.e. over a $Sol^3$-manifold).
\end{prop}

Every closed oriented $4$-manifold that possesses a solvable non-product geometry is a mapping torus:

\begin{thm}[\normalfont{\cite[Sections 8.6 and 8.7]{Hillman}}]\label{t:mappingtorisolvable} \
 \begin{itemize}
  \item[\normalfont{(1)}] A closed $Sol_0^4$- or $Sol_{m \neq n}^4$-manifold is a mapping torus of a self-homeomorphism of $T^3$.
  \item[\normalfont{(2)}] A closed oriented $Nil^4$- or $Sol_1^4$-manifold is a mapping torus of a self-homeomorphism of a $Nil^3$-manifold.
 \end{itemize}
\end{thm}

We note that non-orientable closed $Nil^4$- or $Sol_1^4$-manifolds are not mapping tori of $Nil^3$-manifolds~\cite[Theorem
8.9]{Hillman}. 

\medskip 

\noindent{{\bf Hyperbolic geometries.}} There exist two aspherical irreducible symmetric geometries, namely the real and the complex hyperbolic, denoted
by $\mathbb{H}^4$ and $\mathbb{H}^2(\mathbb{C})$ respectively. We will not be dealing with these geometries in Theorem \ref{t:order4}; 
see Section \ref{ss:hyper} for a brief discussion.

\medskip

A crucial property for our study is that the $4$-dimensional geometries are homotopically unique, by the following result of Wall:

\begin{thm}[\normalfont{\cite[Theorem 10.1]{Wall:geom4-mfds2}, \cite[Prop. 1]{Kotschick:4-mfds}}]\label{t:Wall4Dgeometries}
 If $M$ and $N$ are homotopy equivalent closed $4$-manifolds possessing geometries $\mathbb{X}^4$ and $\mathbb{Y}^4$ respectively, then
$\mathbb{X}^4 = \mathbb{Y}^4$.
\end{thm}

In particular, {\em a closed aspherical geometric $4$-manifold $M$ is finitely covered by a closed $\mathbb{X}^4$-manifold if and only if it possesses the geometry $\mathbb{X}^4$}.

\section{4-manifolds covered by products}\label{s:products}

In this section we deal with maps between closed $4$-manifolds that are virtually direct products.

\subsection{Non-existence stability between products}

We first recall from~\cite{Neothesis,KotschickLoehNeofytidis} some general results about maps between direct products.

Following the discussion of Section \ref{s:Wang}, a natural question is whether one can extend Wang's ordering given by Theorem \ref{t:order3-mfds} to $4$-manifolds that are finitely covered by $N\times S^1$, where $N$ is a $3$-manifold as in Theorem \ref{t:order3-mfds}. The non-trivial problem is to extend the non-existence results of Wang's ordering. A priori, it is not clear whether $M\ngeq N$ implies $M\times S^1 \ngeq N\times S^1$. More generally, a natural question is whether $M\ngeq N$ is stable under taking direct products, that is, whether $M\ngeq N$ implies $M\times W \ngeq N\times W$ for every manifold $W$. This question has its own independent interest, because, for example, our current understanding of the multiplicativity of functorial semi-norms (such as the simplicial volume) under taking products is not sufficient enough to provide answers to this kind of problems, even when a semi-norm remains non-zero under taking products.

The following result gives a sufficient condition for non-domination stability under taking direct products:

\begin{thm}[\cite{KotschickLoehNeofytidis,Neothesis}]\label{t:mapsbetweenproducts}
 Let $M$, $N$ be closed oriented $n$-dimensional manifolds such that $N$ is not dominated by products and $W$ be a closed oriented manifold of dimension $m$. Then, $M \geq N$ if and only if $M \times W \geq N \times W$.
\end{thm}

The proof of the above statement is based on the celebrated realization theorem of Thom \cite{Thom}; see~\cite{KotschickLoehNeofytidis,Neothesis} for details. In the same spirit, we obtain the following:

\begin{prop}[\cite{KotschickLoehNeofytidis,Neothesis}]\label{c:productslower}
 Let $M$, $W$ and $N$ be closed oriented manifolds of dimensions $m$, $k$ and $n$ respectively such that $m,k<n<m+k$. If $N$ is not dominated by
products, then $M\times W\geq N\times V$ for no closed oriented manifold $V$ of dimension $m+k-n$.
\end{prop}

\subsection{Targets that are virtual products with a circle factor.}

Now we apply Theorem \ref{t:mapsbetweenproducts} to closed $4$-manifolds that are finitely covered by products of type $N \times S^1$. The main result
of this subsection extends the ordering of Theorem \ref{t:order3-mfds} as follows:

\begin{thm}\label{t:4Dwangorder}
Let $X$ be one of the three classes $\mathrm{(i)-(iii)}$ of Theorem \ref{t:order3-mfds}. We say that a closed $4$-manifold belongs to the class $X
\times \R$ if it is finitely covered by a product $N \times S^1$, where $N$ is a closed $3$-manifold in the class $X$.

If there exists an oriented path from the class $X$ to the class $Y$ in Figure \ref{f:order3-mfds}, then any closed $4$-manifold in the class $Y \times
\R$ is dominated by a manifold of the class $X \times \R$. If there is no oriented path from the class $X$ to the class $Y$, then no manifold in the
class $Y \times \R$ can be dominated by a manifold of the class $X \times \R$.
\end{thm}

{\noindent{\em Proof of existence.}} 
The existence part follows easily by the corresponding existence results for maps between $3$-manifolds given in Theorem \ref{t:order3-mfds}. Namely, let $Z$ be a closed $4$-manifold in the class $Y \times \R$ and suppose that there is an arrow from $X$ to $Y$ in Figure \ref{f:order3-mfds}. By definition, $Z$ is finitely covered by a product $N \times S^1$ for some closed $3$-manifold $N$ in the class $Y$. By Theorem \ref{t:order3-mfds}, there is a closed $3$-manifold $M$ in the class $X$ and a map of non-zero degree $f \colon M \longrightarrow N$. Then $f \times \mathrm{id}_{S^1} \colon M \times S^1 \longrightarrow N \times S^1$ has degree $\deg(f)$ and the product $M \times S^1$ belongs to the class $X \times \R$.

\medskip
{\noindent{\em Proof of non-existence.}}
We now prove the non-existence part of Theorem \ref{t:4Dwangorder}. Obviously, there is no $4$-manifold in the class $(\#_p S^2 \times S^1) \times \R$ that can dominate a manifold of the other classes. Thus, the interesting cases are when both the domain and the target are aspherical.

We first deal with targets whose $3$-manifold factor $N$ in their finite cover $N \times S^1$ is not dominated by products.

\begin{prop}\label{p:liftingtoproducts}
 Let $W$ and $Z$ be two closed oriented $4$-manifolds. Suppose that 
\begin{itemize}
 \item[\normalfont{(1)}] $W$ is dominated by products, and
 \item[\normalfont{(2)}] $Z$ is finitely covered by a product $N \times S^1$, where $N$ is a closed oriented $3$-manifold which is not dominated by products.
\end{itemize}
 If $W \geq Z$, then there exists a closed oriented $4$-manifold $M \times S^1$ so that $M \times S^1 \geq W$ and $M \geq N$. In particular, $M$ cannot be
dominated by products.
\end{prop}
\begin{proof}
Assume that $f \colon W \longrightarrow Z$ is a map of non-zero degree and $p \colon N \times S^1 \longrightarrow Z$ is a finite covering of $Z$, where $N$ is a closed oriented $3$-manifold that is not dominated by products. The intersection
\[
 H := \mathrm{im} (\pi_1(p)) \cap \mathrm{im} (\pi_1(f)) 
\]
is a finite index subgroup of $\mathrm{im} (\pi_1(f))$ and its preimage $G := \pi_1(f)^{-1}(H)$ is a finite index subgroup of $\pi_1(W)$. Let $p'\colon \overline{W} \longrightarrow W$ be the finite covering of $W$ corresponding to $G$ and $\bar{f} \colon \overline{W} \longrightarrow N \times S^1$ be the lift of $f \circ p'$. 

By assumption, there is a non-trivial product $P$ and a dominant map $g \colon P \longrightarrow \overline{W}$. Thus, we obtain a non-zero degree map $\bar{f} \circ g \colon P \longrightarrow N \times S^1$. Now, since $P$ is a $4$-manifold, there exist two possibilities: Either $P = M\times S^1$, for a closed oriented $3$-manifold $M$ or $P = \Sigma_g \times \Sigma_h$, where $\Sigma_g$ and $\Sigma_h$ are closed oriented hyperbolic surfaces of genus $g$ and $h$ respectively. The latter possibility is excluded by Proposition \ref{c:productslower}, because $N$ is not dominated by products. Thus $P = M \times S^1$, and so we obtain a non-zero degree map $M \times S^1 \longrightarrow N \times S^1$. Then $M \geq N$ by Theorem \ref{t:mapsbetweenproducts}, again because $N$ is not dominated by products. Clearly, $M$ cannot be dominated by products.
\end{proof}

\begin{cor}
If $Y \neq \mathbb{H}^2 \times \R$ or $\R^3$, then the non-existence part of Theorem \ref{t:4Dwangorder} holds true for every aspherical target in a class $Y \times \R$.
 \end{cor}
\begin{proof}
 By~\cite[Theorem 4]{KotschickNeofytidis}, the only closed aspherical $3$-manifolds that are dominated by products are those carrying one of the geometries $\mathbb{H}^2 \times \R$ or $\R^3$. The corollary now follows by Proposition \ref{p:liftingtoproducts} and the non-existence result in dimension three given by Theorem \ref{t:order3-mfds}.
\end{proof}

In terms of $4$-dimensional geometries of type $\mathbb{X}^3 \times \R$ we obtain the following straightforward consequence:

\begin{cor}
 Suppose that $W$ and $Z$ are closed oriented aspherical $4$-manifolds carrying product geometries $\mathbb{X}^3 \times \R$ and $\mathbb{Y}^3 \times \R$ respectively. Assume that $\mathbb{Y}^3$ is not $\mathbb{H}^2 \times \R$ or $\R^3$. If $W \geq Z$, then every closed $\mathbb{Y}^3$-manifold is dominated by a closed $\mathbb{X}^3$-manifold.
\end{cor}

In order to complete the proof of Theorem \ref{t:4Dwangorder}, we need to show that closed manifolds which belong to the classes $\mathbb{H}^2 \times \R^2$ or $\R^4$
are not dominated by closed manifolds of the classes $Sol^3 \times \R$, $\widetilde{SL_2} \times \R$ or $Nil^3 \times \R$.

Since the first Betti numbers of closed $Sol^3 \times \R$-manifolds are at most two, and of closed $Nil^3 \times \R$-manifolds at most three, such manifolds cannot dominate closed manifolds possessing one of the geometries $\mathbb{H}^2 \times \R^2$ or $\R^4$.

Finally, we deal with the $\widetilde{SL_2} \times \R$ geometry:

\begin{lem}
 There is no closed oriented $\widetilde{SL_2} \times \R$-manifold that can dominate a closed oriented manifold possessing one of the geometries  $\mathbb{H}^2 \times \R^2$ or $\R^4$.
\end{lem}
\begin{proof}
Every closed $\R^4$-manifold is finitely covered by $T^4$ and, therefore, is virtually dominated by every closed $\mathbb{H}^2 \times \R^2$-manifold. Thus, it suffices to show that $T^4$ cannot be dominated by a product $M \times S^1$, where $M$ is a closed  $\widetilde{SL_2}$-manifold. After passing to a finite cover, we can assume that $M$ is a non-trivial circle bundle over a hyperbolic surface $\Sigma$~\cite{Scott:3-mfds}.

Suppose that $f \colon M \times S^1 \longrightarrow T^4$ is a continuous map. The product $M \times S^1$ carries the structure of a non-trivial circle bundle over $\Sigma \times S^1$, by multiplying by $S^1$ both the total space $M$ and the base surface $\Sigma$ of the circle bundle $M \longrightarrow \Sigma$. The $S^1$-fiber of the circle bundle 
\[
S^1\longrightarrow M \times S^1\longrightarrow \Sigma \times S^1
\]
has finite order in $H_1(M\times S^1)$, being also the fiber of $M$. Therefore, its image under $H_1(f)$ has finite order in $H_1(T^4)$. Now, since $H_1(T^4)$ is isomorphic to $\pi_1(T^4) \cong \Z^4$, we deduce that $\pi_1(f)$ maps the fiber of the circle bundle $M \times S^1 \longrightarrow \Sigma \times S^1$ to the trivial element in $\pi_1(T^4)$. The latter implies that $f$ factors through the base $\Sigma \times S^1$, because the total space $M \times S^1$, the base $\Sigma \times S^1$ and the target $T^4$ are all aspherical. This finally means that the degree of $f$ must be zero, completing the proof.
\end{proof}

We have now finished the proof of Theorem \ref{t:4Dwangorder}.

\subsection{Virtual products of two hyperbolic surfaces.}

We close this section by examining manifolds that are finitely covered by a product of two closed hyperbolic surfaces, i.e. closed reducible $\mathbb{H}^2 \times \mathbb{H}^2$-manifolds.

\medskip
\noindent{\em Reducible $\mathbb{H}^2 \times \mathbb{H}^2$-manifolds as domains.} It is clear that every closed $4$-manifold with geometry modeled on $\mathbb{H}^2 \times \R^2$ or $\R^4$ is dominated by a product of two hyperbolic
surfaces. This implies moreover that every target in the class $\#_p (S^2 \times S^1) \times \R$ is dominated by such a product (by Theorem \ref{t:order3-mfds}). However, as we have seen in the proof of Proposition \ref{p:liftingtoproducts} (see also Proposition \ref{c:productslower}), closed aspherical $4$-manifolds that are virtual products $N \times S^1$, where $N$ does not belong to one of the classes $\mathbb{H}^2 \times \R$ or $\R^3$, cannot be dominated by reducible $\mathbb{H}^2\times \mathbb{H}^2$-manifolds.

\medskip
\noindent{\em Reducible $\mathbb{H}^2 \times \mathbb{H}^2$-manifolds as targets.} 
We claim that there is no manifold in the classes $X \times \R$ which can dominate a product of two closed hyperbolic surfaces. This is obvious when $X = \#_p (S^2 \times S^1)$. When $X$ is a class of aspherical $3$-manifolds, then the technique of factorizing dominant maps applies: The fundamental group of a product $M \times S^1$ has center at least infinite cyclic, whereas the center of the fundamental group of a product of two hyperbolic surfaces $\Sigma_g \times \Sigma_h$ is trivial. Therefore, every ($\pi_1$-surjective) map $f \colon M \times S^1 \longrightarrow \Sigma_g \times \Sigma_h$ kills the homotopy class of the $S^1$ factor of $ M \times S^1$, and so it factors through an aspherical manifold of dimension at most three, because both $M \times S^1$ and $\Sigma_g \times \Sigma_h$ are aspherical. This means that $H_4(f)([M \times S^1]) = 0 \in H_4(\Sigma_g \times \Sigma_h)$, implying that the degree of $f$ is zero.

\begin{rem}\label{r:othertech}
Since $\Sigma_g$ and $\Sigma_h$ are hyperbolic, the conclusion that $M \times S^1 \ngeq \Sigma_g \times \Sigma_h$ is straighforward, because $M \times S^1$ has vanishing simplicial volume, whereas the simplicial volume of $\Sigma_g \times \Sigma_h$ is positive. 
However, we prefer to give more elementary and uniform arguments for the proof of Theorem \ref{t:order4}, following simultaneously our methodology.
\end{rem}

\section{Ordering the non-hyperbolic geometries}\label{s:proofend}

In this section we finish the proof of Theorem \ref{t:order4}. 

The proof for the right-hand side of the diagram in Figure \ref{d:nonhypmaps}, concerning maps between geometric aspherical $4$-manifolds that are virtual products, was obtained in the previous section.

We now deal with the remaining geometries and complete the proof of Theorem \ref{t:order4}. The claim indicated in Figure \ref{d:nonhypmaps} is that
each of the geometries $Nil^4$, $Sol_0^4$, $Sol_{m \neq n}^4$, $Sol_1^4$ and the irreducible geometry $\mathbb{H}^2 \times \mathbb{H}^2$ is not comparable with any other
(non-hyperbolic) geometry under the domination relation.

\subsection{Non-product solvable geometries}

We begin by showing that there are no maps of non-zero degree between any two closed manifolds possessing a different geometry among $Nil^4$, $Sol_0^4$,
$Sol_{m \neq n}^4$ and $Sol_1^4$.

First, we show that there are no maps of non-zero degree between closed $Nil^4$-manifolds and $Sol_1^4$-manifolds. We need the following lemma:

\begin{lem}\label{l:factorizationbundles}
 For $i = 1,2$ let $M_i \stackrel{p_i}\longrightarrow B_i$ be circle bundles over closed oriented aspherical manifolds $B_i$ of the same dimension, so that the
center of each $\pi_1(M_i)$ remains infinite cyclic in finite covers. If $B_1 \ngeq B_2$, then $M_1 \ngeq M_2$. 
\end{lem}
\begin{proof}
 Suppose that $f \colon M_1 \longrightarrow M_2$ is a map of non-zero degree. After passing to finite coverings, if necessary, we may assume that $f$ is
$\pi_1$-surjective and that the center of each $\pi_1(M_i)$ is infinite cyclic. 
Let $p_2\circ f\colon M_1\longrightarrow B_2$. The induced homomorphism $\pi_1(p_2 \circ f)$ maps the infinite cyclic group generated by
the circle fiber of $M_1$ trivially in $\pi_1(B_2)$. This implies that $p_2 \circ f$ factors through the bundle projection $p_1 \colon M_1 \longrightarrow
B_1$ (recall that $B_2$ is aspherical). In particular, there is a continuous map $g \colon B_1 \longrightarrow B_2$, so that $p_2 \circ f = g \circ p_1$ (in homotopy).
Now $f$ factors through the pullback of $M_2\stackrel{p_2}\longrightarrow B_2$ under $g$, which means that the degree of $f$ is a multiple of $\deg (g)$. However, the degree of $g$ is zero by
our hypothesis that $B_1 \ngeq B_2$, and so $\deg(f)=0$. This contradiction finishes the proof.
\end{proof}

Closed $Nil^4$-manifolds and $Sol_1^4$-manifolds are virtual circle bundles over $Nil^3$-manifolds and $Sol^3$-manifolds respectively, and the center of their fundamental groups remains infinite cyclic in finite covers; cf. Prop. \ref{p:nil4} and \ref{p:sol1} respectively. Since there are no maps of non-zero degree between closed $Sol^3$-manifolds and $Nil^3$-manifolds (cf. Theorem \ref{t:order3-mfds}), Lemma \ref{l:factorizationbundles} implies the following:

\begin{prop}
 Closed oriented $Nil^4$-manifolds are not comparable under $\geq$ with closed oriented $Sol_1^4$-manifolds. 
\end{prop}

Next, we show that there are no dominant maps between closed $Sol_0^4$-manifolds and $Sol_{m \neq n}^4$-manifolds. Recall that a closed manifold modeled on $Sol_0^4$ or $Sol_{m \neq n}^4$ is a mapping torus of $T^3$ (Theorem \ref{t:mappingtorisolvable} (1)). Moreover, the eigenvalues of the automorphism
of $\Z^3$ induced by the monodromy of $T^3$ are not roots of unity; cf.~\cite[pg. 164/165]{Hillman}. 

In the following proposition we show that every non-zero degree map between such mapping tori is $\pi_1$-injective: 

\begin{prop}\label{p:injectivemappingtoriTn}
 Let $M$ and $N$ be closed manifolds that are finitely covered by mapping tori of self-homeomorphisms of $T^n$ so that no eigenvalue of the induced
automorphisms of $\Z^n$ is a root of unity. If $f \colon M \longrightarrow N$ is a non-zero degree map, then $f$ is $\pi_1$-injective.
\end{prop}
\begin{proof}
 Since we want to show that $f \colon M \longrightarrow N$ is $\pi_1$-injective, we may write 
 \[
 \pi_1(M) = \pi_1(T^n) \rtimes_{\theta_M} \langle t \rangle,
 \]
where $\pi_1(T^n) = \Z^n = \langle x_1,...,x_n \vert \ [x_i,x_j] = 1 \rangle$ and the automorphism $\theta_M \colon \Z^n
\longrightarrow \Z^n$ is induced by the action of $\langle t \rangle$ on $\Z^n$, given by
\[
tx_it^{-1} = x_1^{k_{1i}} \cdots x_n^{k_{ni}}, \ \mathrm{for \ all} \ i = 1,..., n.
\]
(That is, the matrix of the automorphism $\theta_M$ is given by $(k_{ij}), \ i,j \in \{1,...,n\}$.)
We observe that
$tx_it^{-1} \neq x_j$, for all
$i,j \in \{1,...,n\}$, because no eigenvalue of $\theta_M$ is a root of unity. 

The image $f_*(\pi_1(M))$ of the induced homomorphism $f_* \colon \pi_1(M) \longrightarrow \pi_1(N)$ is a finite index
subgroup of $\pi_1(N)$, 
generated by $f_*(x_1),...,f_*(x_n), f_*(t)$. Also, the relations $[x_i,x_j] = 1$ and $tx_it^{-1} = x_1^{k_{1i}} \cdots x_n^{k_{ni}}$ in $\pi_1(M)$ give
the corresponding relations $[f_*(x_i),f_*(x_j)] = 1$ and $f_*(t)f_*(x_i)f_*(t)^{-1} = f_*(x_1)^{k_{1i}} \cdots f_*(x_n)^{k_{ni}}$ in $f_*(\pi_1(M))$. 

Since $\pi_1(N)$ (and therefore $f_*(\pi_1(M))$) is torsion-free and (virtually) a semi-direct product $\Z^n \rtimes \Z$, where the eigenvalues of the
induced automorphism of $\Z^n$ are
not roots of unity, we conclude that there no other relations between the generators $f_*(x_1),...,f_*(x_n),f_*(t)$ and that $f_*(t)f_*(x_i)f_*(t)^{-1} \neq
f_*(x_j)$,
for all $i,j \in \{1,...,n\}$. Therefore, $f_*(\pi_1(M))$ has a presentation
\begin{eqnarray*}
 f_*(\pi_1(M))  & = & \langle f_*(x_1),...,f_*(x_n),f_*(t) \vert \ [f_*(x_i),f_*(x_j)] = 1,\\
                &   & \ f_*(t)f_*(x_i)f_*(t)^{-1} = f_*(x_1)^{k_{1i}} \cdots f_*(x_n)^{k_{ni}} \rangle\\
                & = & \langle f_*(x_1),...,f_*(x_n) \rangle \rtimes \langle f_*(t) \rangle.
\end{eqnarray*}
In particular, $f_*\vert_{\pi_1(T^n)}$ surjects onto $\pi_1(T^n) \cong \langle f_*(x_1),...,f_*(x_n) \rangle \subset f_*(\pi_1(M))$. Since $\Z^n$ is
Hopfian, we deduce that $f_*\vert_{\pi_1(T^n)}$ is injective. 

Finally, we observe that $f_*(t^k) \notin \langle f_*(x_1),...,f_*(x_n) \rangle$ for all non-zero integers $k$, otherwise the finite index subgroup 
$\langle f_*(x_1),...,f_*(x_n) \rangle \rtimes \langle f_*(t^k) \rangle \subset \pi_1(N)$ would be isomorphic to $\Z^n$, which is impossible. This
completes the proof. 
\end{proof}

Since the $4$-dimensional geometries are homotopically unique (cf. Theorem \ref{t:Wall4Dgeometries}), we deduce the following:

\begin{cor}\label{c:nonmapsbetweensol}
 Any two closed oriented manifolds $M$ and $N$ possessing the geometries $Sol_{m \neq n}^4$ and $Sol_0^4$ respectively are not comparable under
$\geq$.
\end{cor}

Now, we show that closed $Sol_1^4$-manifolds are not comparable with closed manifolds possessing one of the geometries $Sol_{m \neq n}^4$ or $Sol_0^4$.

\begin{prop}
 Closed oriented manifolds possessing the geometry $Sol_1^4$ are not dominated by closed oriented $Sol_{m \neq n}^4$- or $Sol_0^4$-manifolds. Conversely,
closed oriented $Sol_1^4$-manifolds cannot dominate closed oriented manifolds with geometries modeled on $Sol_{m \neq n}^4$ or $Sol_0^4$.
\end{prop}
\begin{proof}
  Let $Z$ be a closed oriented $Sol_1^4$-manifold. By Theorem \ref{t:mappingtorisolvable} (2), $Z$ is a mapping torus of a self-homeomorphism of
a closed $Nil^3$-manifold $N$. However, $Z$ is not a mapping torus of a self-homeomorphism of $T^3$; cf.~\cite[Section 8.6]{Hillman}.

Suppose that there is a non-zero degree map $f \colon W \longrightarrow Z$, where $W$ is a closed oriented $Sol_{m \neq n}^4$- or $Sol_0^4$-manifold.
By Theorem \ref{t:mappingtorisolvable} (1), $W$ is a mapping torus of a
self-homeomorphism of $T^3$ and $\pi_1(W) = \Z^3 \rtimes_{\theta_W} \Z = \langle x_1,x_2,x_3 \rangle \rtimes_{\theta_W} \langle t  \rangle$, where
$\theta_W$ is the automorphism of $\Z^3$ induced by the action by conjugation by $t$. Now $f_*(\pi_1(W))$ has finite index in $\pi_1(Z)$ and $\langle
f_*(t) \rangle$ acts
by conjugation (by $f_*(t)$) on $\langle f_*(x_1),f_*(x_2),f_*(x_3) \rangle$, that is $f_*(\pi_1(W))$ is a semi-direct product
$\langle f_*(x_1),f_*(x_2),f_*(x_3) \rangle \rtimes \langle f_*(t) \rangle$ (recall also that our groups are torsion-free).
However, the generators
$f_*(x_1),f_*(x_2),f_*(x_3)$ commute with each other, contradicting the fact that $\pi_1(Z)$ cannot be (virtually) $\Z^3 \rtimes \Z$. Therefore $W \ngeq
Z$.

For the converse, we have that a closed $Sol_1^4$-manifold $Z$ is finitely covered by a non-trivial circle bundle
over a closed oriented $Sol^3$-manifold and the center of $\pi_1(Z)$ is infinite cyclic (Prop. \ref{p:sol1}). Moreover, the fundamental group of every closed $Sol_{m \neq n}^4$- or $Sol_0^4$-manifold $W$ is not
presentable by products (Prop. \ref{p:sol&sol}). In particular, every finite index subgroup of $\pi_1(W)$ has trivial
center, because $W$ is aspherical. Using the asphericity of our geometries and applying a standard factorization argument we derive that $Z \ngeq W$,
because a dominant map $Z \longrightarrow W$ would factor through the base ($Sol^3$-manifold) of the domain.
\end{proof}

Finally, it has remained to show that there are no dominant maps between closed $Nil^4$-manifolds and closed manifolds possessing one of the geometries
$Sol_{m \neq n}^4$ or $Sol_0^4$.

\begin{prop}
 Closed oriented $Nil^4$-manifolds are not comparable under the domination relation with closed oriented manifolds carrying one of the geometries
$Sol_{m \neq n}^4$ or $Sol_0^4$.
\end{prop}
\begin{proof}
 A closed oriented $Nil^4$-manifold cannot be dominated by closed
$Sol_{m \neq n}^4$- or $Sol_0^4$-manifolds, because the latter have first Betti number one, whereas closed $Nil^4$-manifolds have virtual first
Betti number two; see \cite[Section 6.3.3]{NeoIIPP} for presentations of the fundamental groups of the above manifolds.

Conversely, let $M$ be a closed $Nil^4$-manifold. By Prop. \ref{p:nil4}, $M$ is virtually a non-trivial circle bundle over a closed oriented
$Nil^3$-manifold and $C(\pi_1(M))$ remains infinite cyclic in finite covers. As before, since the fundamental group of every closed $Sol_{m \neq n}^4$- or $Sol_0^4$-manifold $N$
is not
presentable by products (Prop. \ref{p:sol&sol}), and in particular has trivial center, we deduce that every
($\pi_1$-surjective) map $M \stackrel{f}\longrightarrow N$ factors through the base manifold of the domain (which is a $Nil^3$-manifold). This implies that $\deg(f) = 0$ and so
$M \ngeq N$.
\end{proof}

\subsection{Non-product solvable manifolds vs virtual products}

Next, we show that there are no maps of non-zero degree between a closed manifold possessing one of the geometries $Nil^4$, $Sol_0^4$, $Sol_{m \neq n}^4$
or $Sol_1^4$ and a closed manifold carrying a product geometry $\mathbb{X}^3 \times \R$ or the reducible $\mathbb{H}^2 \times \mathbb{H}^2$ geometry.

Using the property ``not (infinite-index) presentable by products", the following result is an application in~\cite{NeoIIPP}:

\begin{thm}{\normalfont(\cite[Theorem F]{NeoIIPP}).}\label{t:productssolvable}
A closed oriented aspherical geometric $4$-manifold $M$ is dominated by a non-trivial product if and only if it is finitely covered by a product. Equivalently, $M$ carries one of the product geometries $\mathbb{X}^3 \times\R$ or the reducible $\mathbb{H}^2 \times \mathbb{H}^2$ geometry.
\end{thm}

Thus closed $4$-manifolds possessing a non-product solvable geometry are not dominated by products. We
therefore only need to show the converse, namely, that solvable manifolds do not dominate manifolds modeled on one of the geometries $\mathbb{X}^3 \times \R$ or the reducible $\mathbb{H}^2
\times \mathbb{H}^2$ geometry. 
First, we deal with nilpotent domains.

\begin{prop}
 A closed oriented $Nil^4$-manifold does not dominate any closed manifold possessing a geometry $\mathbb{X}^3 \times \R$ or the reducible $\mathbb{H}^2
\times \mathbb{H}^2$ geometry.
\end{prop}
\begin{proof}
 Let $W$ be a closed oriented $Nil^4$-manifold. The Abelianization of $\pi_1(W)$ shows that $W$ has virtual first Betti number at most two (cf. Prop. \ref{p:nil4}), and therefore it cannot dominate any closed manifold carrying one of the geometries $\R^4$, $\mathbb{H}^2 \times \R^2$, the reducible $\mathbb{H}^2 \times
\mathbb{H}^2$ geometry, $\widetilde{SL_2} \times \R$, $Nil^3 \times \R$ or $\mathbb{H}^3 \times \R$.  
(The proof for
the $\mathbb{H}^3 \times \R$ geometry follows by the establishment of the virtual Haken conjecture by Agol. 
Nevertheless, we do not actually need to appeal to this deep result because the argument below for the $Sol^3 \times \R$ geometry applies as well for targets possessing the
geometry $\mathbb{H}^3 \times \R$.)

Next, we show that $W$ does not dominate closed $Sol^3 \times \R$-manifolds. Suppose, for contrast, that there is a non-zero degree map $f \colon  W \longrightarrow Z$, where $Z$ is a closed $4$-manifold possessing the geometry $Sol^3 \times \R$. After
passing to finite coverings, if
necessary, we may assume that $f$ is $\pi_1$-surjective, $W$ is a non-trivial circle bundle over a closed oriented $Nil^3$-manifold $M$ (cf. Prop. \ref{p:nil4}) and $Z = N \times S^1$, where $N$ is a closed oriented $Sol^3$-manifold (cf. Theorem
\ref{t:hillmanproducts}). If $p_1
\colon Z \longrightarrow N$ denotes the projection to $N$, then $\pi_1(p_1 \circ f) \colon \pi_1(W) \longrightarrow \pi_1(N)$ kills the $S^1$-fiber of
$W$, because the fundamental group of $N$ has trivial center~\cite{Scott:3-mfds}. Since our spaces are aspherical, we deduce that $p_1 \circ f$ factors through the bundle map $W
\stackrel{p}\longrightarrow M$. However, $H^2(p;\Q)$ is the zero homomorphism because $W \stackrel{p}\longrightarrow M$ is a non-trivial circle bundle. This contradicts the fact that $H^2(p_1 \circ f;\Q)$ is not trivial,
and therefore $W\ngeq Z$. 
(Alternatively, the result follows by Lemma \ref{l:factorizationbundles}, because $M \ngeq N$ according to Theorem \ref{t:order3-mfds}.) 
\end{proof}

Finally, the proof of the following statement is straightforward, because the first Betti number of every closed $Sol_0^4$-,
$Sol_{m \neq n}^4$- or $Sol_1^4$-manifold is one (by the corresponding presentation of their fundamental group; cf. \cite[Sections 8.6 and 8.7]{Hillman} and \cite[Section 6]{NeoIIPP}), and therefore the second Betti number of those manifolds is zero (recall that the
Euler characteristic of those manifolds is zero because they are virtual mapping tori).

\begin{prop}
 A closed oriented manifold possessing one of the geometries $Sol_0^4$, $Sol_{m \neq n}^4$ or $Sol_1^4$ cannot dominate a closed manifold carrying a
geometry $\mathbb{X}^3 \times \R$ or the reducible $\mathbb{H}^2 \times \mathbb{H}^2$ geometry.
\end{prop}
 
\subsection{The irreducible $\mathbb{H}^2 \times \mathbb{H}^2$ geometry}\label{ss:irreducibleH2xH2}

We finally deal with irreducible closed $\mathbb{H}^2 \times \mathbb{H}^2$-manifolds. We show that they cannot be compared under $\geq$ with any other
closed manifold possessing a non-hyperbolic aspherical geometry.

Let $M$ be a closed oriented irreducible $\mathbb{H}^2 \times \mathbb{H}^2$-manifold. Suppose that $f \colon M \longrightarrow N$ is a map of non-zero
degree, where $N$ is a closed aspherical manifold not possessing the irreducible $\mathbb{H}^2 \times \mathbb{H}^2$ geometry. As
usual, we can assume that $f$ is a $\pi_1$-surjective map, after possibly passing to a finite cover. Then we obtain a short exact sequence
\[
 1 \longrightarrow \ker (\pi_1(f)) \longrightarrow \pi_1(M) \stackrel{\pi_1(f)}\longrightarrow \pi_1(N) \longrightarrow 1.
\]
By a theorem of Margulis \cite[Theorem IX.6.14]{Margulis}, the kernel $\ker (\pi_1(f))$ must be trivial, meaning that $\pi_1(f)$ is an isomorphism. Since $M$ and $N$ are
aspherical, we deduce that $M$ is homotopy equivalent to $N$, which contradicts Theorem \ref{t:Wall4Dgeometries}. Therefore $M \ngeq N$.

We now show that $M$ cannot be dominated by any other (non-hyperbolic) geometric closed aspherical $4$-manifold $N$. Since $M$ is not dominated by
products, it suffices to show that $M$ cannot be dominated by a closed manifold $N$ possessing one of the geometries $Sol_1^4$, $Nil^4$, $Sol_{m \neq n}^4$
or
$Sol_0^4$. For each of those geometries, $\pi_1(N)$ has a normal
subgroup of infinite index, which is free Abelian of rank one (geometries $Sol_1^4$ and $Nil^4$) or three (geometries $Sol_{m \neq n}^4$ and $Sol_0^4$);
see \cite[Section 6]{NeoIIPP} for details.
If there were a ($\pi_1$-surjective) map of non-zero degree $f \colon N \longrightarrow M$, then by \cite[Theorem IX.6.14]{Margulis} either $f$ would factor through a
lower dimensional aspherical manifold or $\pi_1(M)$ would be free Abelian of finite rank. The latter cases cannot occur and so $N \ngeq M$. 

We have now shown the following:

\begin{prop}
 Closed irreducible $\mathbb{H}^2 \times \mathbb{H}^2$-manifolds are not comparable under $\geq $ with closed non-hyperbolic $4$-manifolds possessing a
different aspherical geometry.
\end{prop}

This finishes the proof of Theorem \ref{t:order4}.

\subsection{A remark about hyperbolic $4$-manifolds}\label{ss:hyper}
 Using mostly standard properties of the fundamental group or the obstruction ``fundamental group not presentable by products'', it is easy to see that
closed hyperbolic $4$-manifolds (real and complex) are not dominated by geometric non-hyperbolic ones. Gaifullin~\cite{Gaifullin} proved that there exist
real hyperbolic closed $4$-manifolds that virtually dominate all closed $4$-manifolds. Finally, complex hyperbolic $4$-manifolds cannot
dominate real hyperbolic ones; see~\cite{CarlsonToledo} and the related references given there.

\section{The domination relation and Kodaira dimensions}\label{s:Kodaira}

In this section we investigate the monotonicity of Kodaira dimensions of low-dimensional manifolds with respect to the existence of maps of non-zero degree.

The Kodaira dimension $\kappa^t$ of a closed oriented surface $\Sigma_g$ of genus $g$ is defined by
\begin{equation*}
\kappa^t(\Sigma_g)= \left\{\begin{array}{ll}
        -\infty, & \text{if } g=0\\
        0, & \text{if } g=1\\
        1, & \text{if } 9\geq2.
        \end{array}\right.
\end{equation*}
Thus $\kappa^t$ is indeed monotone with respect to the domination relation in dimension two: If $\Sigma_g\geq\Sigma_h$, then $\kappa^t(\Sigma_g)\geq\kappa^t(\Sigma_h)$.

Zhang~\cite{Zhang} defined the Kodaira dimension of $3$-manifolds as follows: Divide the eight $3$-dimensional Thurston geometries into four categories assigning a value to each category:

\begin{center}
\begin{tabular}{rl}
$-\infty:$          & $S^3$, $S^2\times\R$ \\
$0:$   & $\R^3$, $Nil^3$, $Sol^3$ \\
 $1:$          & $\mathbb{H}^2\times\R$, $\widetilde{SL_2}$\\
$\frac{3}{2}:$ & $\mathbb{H}^3$.\\         
\end{tabular}
\end{center}
Let $M$ be a closed oriented $3$-manifold. Consider first the Kneser-Milnor prime decomposition of $M$ and then a toroidal decomposition for each prime summand of $M$, such that at the end each piece carries one of the eight geometries with finite volume. We call this a $T$-decomposition of $M$. The Kodaira dimension of $M$ is then defined as follows:

\begin{defn}\label{d:Kodaira3}(\cite{Zhang}).
The Kodaira dimension $\kappa^t$ of a closed oriented $3$-manifold $M$ is
\begin{itemize}
\item[(1)] $\kappa^t(M) = -\infty$, if for any $T$-decomposition each piece belongs to the category $-\infty$;
\item[(2)] $\kappa^t(M) = 0$, if for any $T$-decomposition there is at least one piece in the category $0$, but no pieces in the category $1$ or $\frac{3}{2}$;
\item[(3)] $\kappa^t(M) = 1$, if for any $T$-decomposition there is at least one piece in the category $1$, but no pieces in the category $\frac{3}{2}$;
\item[(4)] $\kappa^t(M) = \frac{3}{2}$, if for any $T$-decomposition there is at least one hyperbolic piece.
\end{itemize}
\end{defn}

The following result of Zhang states that the Kodaira dimension of $3$-manifolds is monotone with respect to the domination relation. We give a proof for the convenience of the reader.

\begin{thm}{\normalfont (\cite[Theorem 1.1]{Zhang}).}\label{t:Zhang}
Let $M,N$ be closed oriented $3$-manifolds. If $M\geq N$, then $\kappa^t(M)\geq\kappa^t(N)$.
\end{thm}
\begin{proof}
We will show that, if $\kappa^t(M)<\kappa^t(N)$, then $M\ngeq N$. Recall that if $M\geq N_1\# N_2$, then $M\geq N_i$ because $N_1\# N_2\geq N_i$ by pinching the summand $N_j\neq N_i$ to a point. (Note that the pinch map $N_1\# N_2\longrightarrow N_i$ is $\pi_1$-surjective.) So we can assume that the target $N$ is prime and belongs to the  category (among (1)--(4)) with the highest Kodaira dimension, according to which category our initial target (which is a connected sum containing $N$) belongs.

First, if $M$ possesses a Thurston geometry, then the claim is an immediate consequence of Theorem \ref{t:order3-mfds}. 

Next, suppose $M$ is not prime. If $M$ belongs to category (1), then clearly it does not dominate any prime manifold in categories (2)--(4) because those manifolds are aspherical (see also Theorem \ref{t:order3-mfds}). If $M$ belongs to category (3), then it does not dominate any prime manifold in category (4), because manifolds with hyperbolic pieces have positive simplicial volume~\cite{Gromov}, whereas every manifold in category (3) has zero simplicial volume (recall that the simplicial volume is additive for connected sums of manifolds of dimension at least three; cf~\cite{Gromov}).

Finally, suppose $M$ belongs to category (2). By the same argument as above (using simplicial volume), $M$ does not dominate any manifold in category (4). Let now the target $N$ be in category (3). We can assume that $N$ (or each toroidal piece of $N$) is a circle bundle over a hyperbolic surface (or orbifold) $\Sigma$. Suppose $M$ dominates $N$. Then, after possibly passing to finite covers, there is a surjection $\pi_1(M)\twoheadrightarrow \pi_1(\Sigma)$. However this is impossible because $\pi_1(M)$ is a free product of (virtually) solvable groups.
\end{proof}

In dimension four, Zhang defined the Kodaira dimension $\kappa^g$ for geometric manifolds:
\begin{defn}\label{d:Kodaira4}(\cite{Zhang}).
The Kodaira dimension $\kappa^g$ of a closed oriented geometric $4$-manifold $M$ is
\begin{itemize}
\item[(1)] $\kappa^g(M)=-\infty$, if $M$ is modeled on one of the geometries $S^4$, $\CP^2$, $S^3\times\R$, $S^2\times S^2$, $S^2\times\R^2$, $S^2\times\mathbb{H}^2$, $Sol_0^4$ or $Sol_1^4$;
\item[(2)] $\kappa^g(M)=0$, if $M$ is modeled on one of the geometries $\R^4$, $Nil^4$, $Nil^3\times\R$, $Sol_{m\neq n}^4$ or $Sol^3\times\R$;
\item[(3)] $\kappa^g(M)=1$, if $M$ is modeled on one of the geometries $\mathbb{H}^2\times\mathbb\R^2$, $\widetilde{SL_2}\times\R$ or $\mathbb{H}^3\times\R$; 
\item[(4)] $\kappa^g(M)=2$, if $M$ is modeled on one of the geometries $\mathbb{H}^2(\mathbb{C})$, $\mathbb{H}^2\times\mathbb{H}^2$ or $\mathbb{H}^4$.
\end{itemize}
\end{defn}

An application of Theorem \ref{t:order4} is that the Kodaira dimension of geometric $4$-manifolds is also monotone with respect to the domination relation, as conjectured in~\cite{Zhang}: 

\begin{thm}{\normalfont (Theorem \ref{t:Kodaira4}).}
Let $M, N$ be closed oriented geometric $4$-manifolds. If $M\geq N$, then $\kappa^g(M)\geq\kappa^g(N)$.
\end{thm}
\begin{proof}
It is clear that manifolds possessing a non-aspherical geometry do not dominate aspherical manifolds. The proof for the non-hyperbolic aspherical geometries follows by Theorem \ref{t:order4}. Finally, we refer to Section \ref{ss:hyper} for the hyperbolic geometries.
\end{proof}

\bibliographystyle{amsplain}

\end{document}